\documentclass[12pt]{article}

\usepackage{amsmath}
\usepackage{amsfonts}
\usepackage{amssymb}
\usepackage{amscd}
\usepackage{amsthm}
\usepackage{bbm}
\usepackage{color}      
\usepackage[linktocpage=true,plainpages=false,pdfpagelabels=false]{hyperref}

\input xypic

\usepackage[matrix,arrow,curve]{xy}

\usepackage{color}

\newcommand{\quash}[1]{}

\newtheorem{defin}{Definition}
\newtheorem{prop}{Proposition}
\newtheorem{nt}{Remark}
\newtheorem{Th}{Theorem}
\newtheorem{lemma}{Lemma}

\newtheorem{example}{Example}

\newfont{\ssdbl}{msbm8}
\newfont{\sdbl}{msbm9}
\newfont{\dbl}{msbm10 at 12pt}

\newcommand{\oo}{{\cal O}}
\newcommand{\ff}{{\cal F}}

\newcommand{\Aut}{\mathop {\rm Aut}}

\newcommand{\tr}{\mathop {\rm Tr}}

\newcommand{\Frac}{\mathop {\rm Frac}}

\newcommand{\dq}{\mathbb{Q}}

\newcommand{\Z}{\dz}

\newcommand{\Ker}{{\rm Ker}\:}

\newcommand{\lrto}{\longrightarrow}

\newcommand{\df}{\mathbb{F}}

\newcommand{\F}{{\bf F}}

\newcommand{\D}{{\cal D}}

\newcommand{\E}{{\cal E}}

\def\Z{{\mathbb Z}}
\def\Q{{\mathbb Q}}

\def\C{{\mathbb C}}
\def\A{{\mathbb A}}

\newcommand{\Div}{{\rm Div}}

\newcommand{\vp}{\varphi}
\newcommand{\m}{{\mathfrak m}}

\begin{document}

\author{
D. V. Osipov, A. N. Parshin
\footnote{Both authors are partially  supported by
Russian Foundation for Basic Research (grants no.~13-01-12420
 ofi\_m2 and no.~14-01-00178a) and by the grant of supporting of Leading Scientific Schools NSh-2998.2014.1.}}

\title{Representations of the discrete Heisenberg group on distribution spaces of two-dimensional local fields}
\date{}

\maketitle
\flushright{\em To Vladimir Petrovich Platonov on his 75th birthday}
\abstract{We study a natural action of the Heisenberg group of integer unipotent matrices of  the third order on distribution space of  a two-dimensional local field for a flag on a two-dimensional scheme.}

\section{Introduction}
Let $X$ be an algebraic curve over a field $k$. Then one can attach a discrete valuation ring ${\cal O}_x$ and a local field $K_x = \Frac ({\cal O}_x)$ to a closed point $x \in  X$.  Thus the valuation group $\Gamma_x =  K_x^*/{\cal O}_x^*   = \Z$ will be the simplest discrete Abelian group canonically  attached to this point of the curve.  Starting from this local group one can define the global group $\Div(X) = \oplus_{x  \in  X } \Gamma_x$.

In arithmetics, when $k = {\mathbb F}_q$, the representation theory of these groups plays an important role in the study of arithmetical invariants of the curve $X$ such as  number of classes or  zeta-function.   These groups act on the functional spaces of the field  $K_x$  and of the adelic group $\A_X$.  All irreducible representations are of dimension $1$ (i.e. characters with values in $\C^*$), and the functional spaces which are infinite-dimensional can be presented as certain integrals (but not direct sums) of these  characters. This construction is very useful if one wants to get an analytical continuation and functional equation of  $L$-functions of the curve $X$ (see~\cite{W1, W2, P3}).

The goal of  this work is to extend this circle of questions to  algebraic surfaces. We will restrict ourselves by the local groups only.  As it was shown in~\cite{P2} and in the previous publications, the local fields will now be the two-dimensional local fields, and the corresponding discrete group is the simplest non-Abelian nilpotent group of  class  $2$:  the Heisenberg group  ${\rm Heis}(3, \Z)$ of  unipotent matrices of order $3$ with integer entries. This group has a very non-trivial  representation theory (see \cite{P1, P2}).

The classical representation theory of unitary representations of localy compact groups  in a Hilbert space  does not  work smoothly at this case. The Heisenberg group is not a group of type I in the von Neumann classification. Thus,  the Heisenberg group has a bad topology in its unitary dual space, has, generally speaking, no characters for irreducible representations, and the decompositions of representations into   integrals of the irreducible ones can be non-unique, see~\cite{D}.  Note that Abelian groups, compact topological groups and semi-simple Lie groups are the groups of type I. In our case, if we change the category of representation spaces and consider instead of the Hilbert spaces the vector spaces of countable dimension and without any topology, then the  situation will be much better.  This approach is modeled onto the theory of smooth representations of reductive algebraic groups over non-Archimedean local fields. They are  so called $l$-groups, and the discrete groups are such groups. Then  the  basic notions of the smooth representation  theory (see, for example,  \cite{B}) can be applied in this case.  The  new approach get for smooth non-unitary representations a moduli space which is a complex manifold.  The characters can be defined as modular forms on this manifold, see~\cite{P1}.  But the non-uniqueness will be still preserved. The category of this kind of representations   is not semi-simple: there are non-trivial $\rm Ext$'s.

In~\cite{OsipPar1, OsipPar2} the authors have developed harmonic analysis on the two-dimensional local fields and have defined the different types of functional spaces on these fields.

In~\cite[\S 5.4 (v)]{P2} the second named author asked a question on the structure of the natural representation of the discrete Heisenberg group
$G={\rm Heis}(3, \Z)$ on
the infinite-dimensional $\C$-vector space $V=\D'_{\oo}(K)^{{\oo'}^*}$. Here  $K$ is a two-dimensional local field, $\oo$ is the discrete valuation ring   of the field $K$,
$\oo'$ is the ring of valuation of rank $2$ of the field $K$, $\D'_{\oo}(K)$ is the space of distributions on $K$ introduced in~\cite{OsipPar1}
 and~\cite{OsipPar2}.

In this paper we will give a partial answer to the above question. In particular, we will describe explicitly  some irreducible subrepresentations of the representation of  $G$
on  $V$ and will relate these representations with classification of certain irreducible representation of $G$ given in~\cite{P1} and \cite{AP}. We note that the group $\tilde{G}=G \rtimes \Z$ also acts  naturally on the space $V$. We will calculate  the traces of  the group $\tilde{G}$ on some irreducible subrepresentations of the representation of $G$ on  $V$. These traces are theta-series, as it was predicted in~\cite[Example]{P1}
by another reasons.

Certainly, the results of this paper are only the first steps in the representation theory of discrete nilpotent groups.
Recently, it was shown that a reasonable class of irreducible representations of any finitely generated nilpotent group coincides with the class of monomial representations, see~\cite{BG}. This was known earlier for finite groups, Abelian groups and nilpotent groups of  class $2$.  The general result was conjectured by the second named author in~\cite{P2}.

 {\bf Acknowledgements.}
The first named author is grateful to the Institut des Hautes \'{E}tudes Scientifiques for the excellent working conditions, since the part of this research was done during his visit to IH\'{E}S in October-November 2013.

\section{Case of one-dimensional local fields} \label{sec2}
Let $L$ be a one-dimensional local field, i.e. a complete discrete valuation  field  with a finite residue field $\df_q$.  It is well-known that
$L \simeq \df_q((u))$, or $L$ is a finite extension of $\mathbb{Q}_p$. Let $\oo_L$ be the discrete valuation ring of $L$,
and $\m_L$ be the maximal ideal of  $\oo_L$.

Let $\D(L)$ be the space of $\mathbb{C}$-valued locally constant functions with compact support on $L$, $\E(L)$  the space of
$\mathbb{C}$-valued
locally constant functions on $L$, and $\D'(L)$ be the space of distributions on $L$, i.e. the dual vector space to the discrete space $\D(L)$.

The additive group of the field $L$ acts on the above spaces by translations.
 Let ${\mu(L) = \D'(L)^{L}}$ be the space of invariant  elements under this action. In other words,  it is the space of $\C$-valued  Haar measures on $L$. We have $\dim_{\C}(\mu(L))=1$.

The group $L^*$ acts on the additive group of the field $L$ by multiplications. Hence we have an action of the group $L^*$ on the the spaces $\D(L)$, $\E(L)$, $\D'(L)$, $\mu(L)$.
We note that the subgroup $\oo_L^*$ acts trivially on the space $\mu(L)$.
Hence the group $\Z \simeq L^*/\oo_L^*$ acts on the space $\mu(L)$ by the character $\Z \ni a  \longmapsto q^a \in \C^* $.  (A local parameter $u$ from $L^*$ goes to  the  element $1$ from $\Z$ via the isomorphism $L^*/\oo_L^*  \simeq \Z$.)

We note that the group $\Z \simeq L^*/\oo_L^*$ acts naturally on  $\C$-vector spaces of invariant elements:  $\D(L)^{\oo_L^*}$, $\E(L)^{\oo_L^*}$, $\D'(L)^{\oo_L^*}$. We have also the following subrepresentations of the representation of the group $\Z$ on the space $\D'(L)^{\oo_L^*}$:
\begin{equation}  \label{form1}
\D(L)^{\oo_L^*} \otimes_{\C}  \mu(L) \;  \subset \; \E(L)^{\oo_L^*} \otimes_{\C} \mu(L) \;  \subset  \; \D'(L)^{\oo_L^*} \mbox{,}
\end{equation}
where the last embedding comes from the map:
$$f \otimes \nu   \longmapsto \left\{ g \mapsto \nu(f g)   \right\} \qquad \forall \quad f \in \E(L)\mbox{,} \quad  g \in \D(L) \mbox{,} \quad \nu \in \mu(L) \mbox{.}$$

An explicit description of the space $\D(L)^{\oo_L^*}$ as the space of some sequences is the following:
$$
\D(L)^{\oo_L^*} = \left\{ \left( c_n \right)_{n \in \Z} \; \left|
\begin{array}{l} c_n \in \C\mbox{,} \quad  c_n =0 \quad \mbox{if} \quad n \ll 0 \mbox{,} \\
c_n=c_{n+1}=c_{n+2}=\ldots \quad \mbox{if} \quad n \gg 0
\end{array} \right.
\right\} \mbox{,}
$$
where by $f \in \D(L)^{\oo_L^*}$ we construct $c_n = f(u^n) \in \C$ for $n \in \Z $ (here $u \in \oo_L$ is a local parameter of the field $L$).

\begin{prop}  \label{dikoe}
The representation of the group $\Z$ on the space $\D(L)^{\oo_L^*}$ does not contain  irreducible subrepresentations.
\end{prop}
\begin{proof}
We prove that for any nonzero $v \in \D(L)^{\oo^*}$ the space $\C[\Z] \cdot v$  contains a proper nonzero $\Z$-invariant subspace. We consider an isomorphism:
$$
\D(L)^{\oo_L^*}   \lrto \C[z,z^{-1}]
\qquad
 \mbox{given as} \qquad \left(c_n\right)_{n \in \Z}   \longmapsto  (1-z) (\sum_{n \in \Z} c_n z^n ) \mbox{.}$$
Under this isomorphism the action of an element $ 1\in \Z$ goes to the multiplication on $z$.
Therefore for any element $w \in \C[z,z^{-1}]$ we have $\C[\Z] \cdot w= \C[z,z^{-1}]w$. Now for any element
$f \in \C[z]$ such that $f \notin \C \cdot z^k$ (for any integer $k$) we have that $\C[\Z] \cdot fw = \C[z,z^{-1}]fw$ is a proper nonzero $\Z$-invariant subspace in the space $\C[z,z^{-1}]w$.
The proposition is proved.
\end{proof}

\vspace{0.3cm}

We note that between  the spaces $\D(L)$ and $\E(L)$ there is a space $\tilde{\E}(L)$ that is the space of uniformly locally constant functions on $L$,
i.e. $f$ from  $\E(L)$ belongs to $\tilde{\E}(L)$ if and only if there is $n_f \in \mathbb{N}$ such that $f(x + y)=f(x)$ for any $x \in L$,
 $y \in \m_L^{n_f}$.  It is clear that ${ \tilde{\E}(L)^{\oo_L^*} = \E(L)^{\oo_L^*}}$ and
$$
\E(L)^{\oo_L^*} = \left\{ \left( c_n \right)_{n \in \Z} \; \left|
\begin{array}{l} c_n \in \C\mbox{,} \quad   \mbox{,} \quad
c_n=c_{n+1}=c_{n+2}=\ldots \quad \mbox{if} \quad n \gg 0
\end{array} \right.
\right\} \mbox{.}
$$
The one-dimensional subspace of all constant functions on $L$ is the set of all eigenvectors in $\E(L)^{\oo_L^*}$ with respect to the action
of $1 \in \Z$.

We note that
$$\D(L) = \mathop{\lim\limits_{\longrightarrow}}\limits_{m \in \Z}  {\ff_0}(m) \mbox{,}$$
where $\ff_0(m)$ is the space of all $\mathbb{C}$-valued
 functions with finite support on $L/ {\mathfrak m}_L^m$. The transition  map  $\ff_0(m_1)  \to \ff_0(m_2)$ for ${m_1 < m_2 }$ is induced by the natural map
 ${L/ {\mathfrak m}_L^{m_2}  \to L/ {\mathfrak m}_L^{m_1}}$, and the map $\ff_0(m) \to \D(L)$ is induced by the map $L \to L/ {\mathfrak m}_L^{m}$.
 Hence we have an equality
 $$\D'(L) = \mathop{\lim\limits_{\longleftarrow}}\limits_{m \in \Z}  {\ff}(m) \mbox{,}$$
 where $\ff(m)$ is the space of all $\mathbb{C}$-valued
 functions on $L/ {\mathfrak m}_L^m$. The transition map ${\ff(m_2)  \to \ff(m_1)}$ for $m_1 < m_2 $ is given by the sums of a function over the elements of finite fibers of the map $L/ {\mathfrak m}_L^{m_2}  \to L/ {\mathfrak m}_L^{m_1}$,
 and the pairing $\D(L)   \times \D'(L)  \lrto \C$ is induced by the pairing $\ff_0(m)  \times \ff(m) \to \C$ for any integer $m$. The last pairing is given by the sum of product of two functions over elements of  $L/ {\mathfrak m}_L^m$. In a slightly another language this presentation of $D'(L)$
 is described, for example, in~\cite[\S~4.2]{OsipPar1}).

Now we can explicitly describe   the space $\D'(L)^{\oo_L^*}$:
\begin{equation}  \label{eq1}
\D'(L)^{\oo_L^*} = \mathop{\lim_{\longleftarrow}}_{m \in \Z}  S(m)   \mbox{,}
\end{equation}
where $S(m)$ is the space of sequences:
$$
S(m) = \left\{ (a_n)_{n \le m} \; \left| \;  a_n \in \C   \right.   \right\}   \mbox{.}
$$
 Any $f \in \ff(m)^{\oo_L^*}$ goes to $(a_n)_{n \le m} \in S(m)$, where $a_n = f(u^m)$ and $u \in \m_L$ is a local parameter.
The  surjective maps ${S(m+1)   \to S(m)}$ in the projective limit of formula~\eqref{eq1} are the following:
$$
(a_n)_{n \le m+1}  \longmapsto (b_n)_{n \le m}   \mbox{,}
$$
where $b_m= (q-1)a_m + a_{m+1}$, $b_{m-k}= q a_{m-k}$ if $k >0$.

The pairing $\D(L)   \times \D'(L)  \lrto \C$ induces the non-degenerate pairing
\begin{equation}  \label{eq2}
\D(L)^{\oo_L^*} \times \D'(L)^{\oo_L^*}  \lrto \C  \mbox{,}
\end{equation}
which can be described explicitly in the following way. We fix an element $(c_n)_{n \in \Z}  \in \D(L)^{\oo_L^*}$ and elements
$(a_n)_{n \le m} \in S(m)$, $m \in \Z$ compatible with projective system~\eqref{eq1}.
We choose any integers $N < M$ such that $c_n =0$ for any $n < N$ and $c_{M+k}=c_M$ for any $k >0$. We define $l=M-N$. Now pairing~\eqref{eq2}
applied to elements
$(c_n)_{n \in \Z}$ and $(a_n)_{n \le M} \in S(M)$ is equal to
$$
(q^l - q^{l-1})c_N a_N + (q^{l-1} - q^{l-2})c_{N+1}a_{N+1} + \ldots + (q-1)c_{M-1}a_{M-1} + c_Ma_M \mbox{.}
$$
This result does not depend on the choice of appropriate $M$ and $N$.

We note that the action of $1 \in \Z$ on $\D'(L)$ maps an element $(a_n)_{n \le m} \in S(m)$, $m \in \Z$ to an element $(b_n)_{n \le m+1} \in S(m+1)$,
$m \in \Z$
where $b_n = a_{n-1}$ for $n \le m+1$.

Using explicit description of the space $\D'(L)^{\oo_L^*}$ by formula~\eqref{eq1}, we obtain a proposition by direct calculations.
\begin{prop} \label{alpha}
We fix any  $\alpha \in \C^*$. We define a $\C$-vector space
$$
\Phi_{\alpha} = \left\{ \phi \in \D'(L)^{\oo_L^*} \;  \left|  \; u(\phi)= \alpha \phi   \right.        \right\} \mbox{,}
$$
where $u \in L^*$ is a local parameter, $\phi \longmapsto  u(\phi) \in \D'(L)^{\oo_L^*} $ is an action of $u \in L^*$ on $\phi$ (which is the same as an action of $1 \in \Z$ on $\phi$).
We have $$\dim_{\C} \Phi_{\alpha} =1 \mbox{.}$$
\end{prop}
\begin{nt}  \label{nt1}
\em
For any $\alpha \in \C^*$ we can construct an element $ g_{\alpha}$ from $\Phi_{\alpha} $ explicitly:
$$
(a_n)_{n \le m}  \in S(m)  \mbox{,}  \quad \mbox{where}  \quad a_m = \alpha^{-m} \mbox{,} \quad a_{m-k} = \frac{\alpha^{-m+k-1}(\alpha-1)}{q^{k-1}(q-1)}
\quad  \mbox{for} \quad k >0  \mbox{.}
$$
We note that the element $g_{\alpha}$ is uniquely defined inside the $1$-dimensional space $\Phi_{\alpha}$
by a property that the pairing of $g_{\alpha}$ with the characteristic function of the subgroup $\oo_L$ is equal to $1$.
\end{nt}
\begin{example} \label{example1} \em
If $\alpha =1$ then $\Phi_{\alpha} = \C \cdot \delta_0$, where $\delta_0$ is the delta-function, i.e. $\delta_0(f)= f(0)$ for any $f \in \D(L)$.
Besides, $g_1 = \delta_0$.

If $\alpha =q$, then $\Phi_{\alpha} = \mu(L) $ is the space of $\C$-valued Haar measures on $L$.
Besides, $g_q$ is the Haar measure which is equal to $1$ after its application to  the open compact subgroup $\oo_L$.
\end{example}

\begin{nt}  {\em The issues of this section were for the first time  formulated and developed  by A.~Weil in \cite{W1} (see also a discussion in \cite[Weil problem]{P3}). }
\end{nt}

\section{Case of two-dimensional local fields}
\subsection{Two-dimensional local fields and  the discrete Heisenberg group} \label{tdl}
Let $K$ be a two-dimensional local field.
Assume that  $K$  is  isomorphic to one of the following fields (we will not consider another types of two-dimensional local fields):
\begin{equation}  \label{cases}
\df_q((u))((t)) \mbox{, } \qquad  \qquad L((t)) \mbox{, }  \qquad  \qquad  M\{\{ u \}\} \mbox{,}
\end{equation}
where the fields $L$ and $M$ are finite  extensions of the field $\Q_p$ such that the finite field $\df_q$ is the residue field of both  $L$ and $M$.
(We recall that the field
$M\{\{u\}\}$  is the completion of the field $\Frac(\oo_M[[u]])$ with respect to the discrete valuation given by the prime ideal $ \m_M[[u]]$ of height $1$
in the ring $\oo_M[[u]]$. Here $ \m_M$ is the maximal ideal of the discrete valuation ring $\oo_M$ of the field $M$.)
 See also, for example, \cite[\S~2]{Osi} for the brief description how above types of two-dimensional local fields are constructed from an algebraic surface over
a finite field or from an arithmetic surface.

Let $u$ be a local parameter of  $L$ (in the second case of~\eqref{cases}), $t$ be a local parameter of  $M$
(in the third case of~\eqref{cases}).

The field $K$ is a discrete valuation field   such that $t$ is a local parameter of this valuation. Let $\oo$ be the discrete
valuation ring of  $K$. We define the $\oo$-module $\oo(i) \subset K$ for any $i \in \Z$ as $\oo(i)= t^i O$. We have $\oo(0)=\oo$,
$\oo(0)/\oo(1)$ is the residue field of $K$. The image of the element $u$ in the residue field is the local parameter there.

The pair $t,u$ is called a system of local parameters of the field $K$.

We define a subring $\oo' \subset \oo$. We say that an element $g $
from  $\oo$ belongs to  $\oo'$ if and only if  the image of $g$ in the residue field of  $K$ belongs
to the discrete valuation ring of the residue field of $K$.

The field $K$ has a  valuation $\nu$ of rank $2$ such that $\nu(f)= (a,b) \in \Z^2$ for $f \in K^*$ if and only if $f= t^a u^b g$
where $g \in {\oo'}^*$. It is clear that the ring $\oo'$ is the valuation  ring of  $\nu$.

In accordance with~\eqref{cases} we define a subring $B$ of $K$ as one of the following expressions:
$$
\df_q[[u]]((t)) \mbox{, } \qquad  \qquad \oo_L((t)) \mbox{, }  \qquad  \qquad  M \cdot \oo_M[[u]] \mbox{,}
$$
where $\oo_L$ is the discrete valuation ring of $L$. (We note that the subring $B$ depends on the choice of  isomorphisms~\eqref{cases}.)

Let $i \le j \in \Z$. An $\oo$-module $\oo(i)/\oo(j)$ is a locally compact group
with the base of neighbourhoods of $0$ given by subgroups $(u^kB \cap \oo(i)) / (u^kB \cap \oo(j))$, $k \in\Z$.
For any $k \in \Z$ the subgroup $ (u^kB \cap \oo(i)) / (u^kB \cap \oo(j))$ is an open compact subgroup in $\oo(i)/ \oo(j)$. We define a $1$-dimensional $\C$-vector space
$\mu(\oo(i)/ \oo(j))$ as the space of $\C$-valued Haar measures on the group $\oo(i)/\oo(j)$. For any $r, s \in \Z $ we define
a $1$-dimensional $\C$-vector space  $\mu(\oo(r) \mid \oo(s)) =  \mu(\oo(s) / \oo(r))$ when $r \ge s $, and
$\mu(\oo(r) \mid \oo(s)) =  \mu(\oo(r) / \oo(s))^*$ when $r \le s$. We have
that $\mu(\oo(r) \mid \oo(r))$ is canonically isomorphic to $\C$. There is also a canonical isomorphism for any elements $r,s, w$ from $\Z$:
\begin{equation}  \label{isom}
\mu (\oo(r) \mid \oo(s)) \otimes_{\C} \mu(\oo(s) \mid \oo(w))  \lrto \mu(\oo(r) \mid \oo(w))
\end{equation}
satisfying associativity for any four elements from $\Z$. Any element $g \in K^*$ maps the $\C$-vector space  $\mu(\oo(r) \mid \oo(s))$ to
the $\C$-vector space $\mu(g\oo(r) \mid g\oo(s)) $.

\begin{lemma} \label{lem1}
The group ${\oo'}^* $ acts by the trivial character on the $1$-dimensional $\C$-vector space $\mu(\oo(s)/\oo(r))$ for any integers $ s \le r$.
\end{lemma}
\begin{proof}
The proof follows by induction on $r-s$ with the help of formula~\eqref{isom}. The case $r-s=1$ is trivial as the case
 of one-dimensional local fields. The lemma is proved.
\end{proof}

We describe now a central extension of $K^*$ by $\C^*$:
\begin{equation}  \label{centext}
1 \lrto \C^* \lrto \widehat{K^*} \lrto K^* \lrto 1 \mbox{,}
\end{equation}
where the group $\widehat{K^*}$ consists from the pairs $(g, \mu)$: $g \in K^*$, $\mu \in \mu(\oo \mid g\oo) )$ and $\mu \ne 0$. The group law in
$\widehat{K^*}$ is the following: $(g_1, \mu_1) \cdot (g_2, \mu_2)= (g_1 g_2, \mu_1 \otimes g_1(\mu_2))$.

Now for any $i \le j \in \Z$
 instead
the $\C^*$-torsor $\mu(\oo(i)  / \oo(j)) \setminus 0$ we consider  a $\Z$-torsor which is a set of Haar measures on the group $\oo(i)/\oo(j)$
with the property that the value of a Haar measure $\mu$ on the open compact subgroup \linebreak
 $ (B \cap \oo(i)) / (B \cap \oo(j))$ is equal to $q^b$ for some
$b \in \Z$ (an integer $b$ depends on $\mu$). Using these $\Z$-torsors, by the same construction as for  central extension~\eqref{centext}
we obtain a  central extension of $K^*$  by $\Z$:
\begin{equation}   \label{centext2}
0 \lrto \Z  \lrto \widetilde{K^*}   \lrto K^* \lrto 1  \mbox{.}
\end{equation}
Besides, we have an embedding of central extension~\eqref{centext2} to central extension~\eqref{centext}: an embedding on kernels
$\Z \ni b \longmapsto q^b \in \C^*$ gives an embedding of groups $\widetilde{K^*}  \hookrightarrow \widehat{K^*}$.

We note that central extensions~\eqref{centext2} and~\eqref{centext} split canonically over the subgroup $\oo^*$ of the group $K^*$:
$g \mapsto (g,1)$, where $1 \in \C= \mu(\oo \mid g\oo)$ since $g\oo=\oo$. Therefore these central extensions split canonically over the subgroup
${\oo'}^*$, i.~e. we can consider the group ${\oo'}^*$ as a subgroup of the group $\widetilde{K^*}$.
 It is easy to see from Lemma~\ref{lem1} that the subgroup
${\oo'}^*$ is  contained in the center of the group $\widehat{K^*}$, and therefore this subgroup is also contained in the center of the group
 $\widetilde{K^*} $.

For any integers $r$ and $s$ we define an element $\mu_{B,r,s} \in \mu(\oo(r) \mid \oo(s))$ as following. If $s \le r$, then $\mu(\oo(r) \mid \oo(s))=
\mu(\oo(s)/\oo(r))$, and the element $\mu_{B,r,s}$ is a Haar measure on the group $\oo(s)/\oo(r)$ such that the value of  $\mu_{B,r,s}$ on the open compact
subgroup
$(B \cap \oo(s))/ (B  \cap \oo(r))$ is equal to $1$. If $s \ge r$, then $\mu(\oo(r) \mid \oo(s))= \mu(\oo(r)/ \oo(s))^*$, and the element
$\mu_{B,r,s}$
is the dual element to the element  $\mu_{B,s,r} \in \mu(\oo(s) \mid \oo(r))= \mu(\oo(r)/\oo(s))$. We note that $\mu_{B,r,r}=1 \in \C$,
and for any integers $r$, $s$ and $w$ the element $\mu_{B,r,s} \otimes \mu_{B,s,w}$ goes to the element $\mu_{B,r,w}$ via
isomorphism~\eqref{isom}.

\bigskip
We consider the discrete Heisenberg group $G = {\rm Heis}(3, \Z) $ which is the group of matrices of the form:
$$
\left(
\begin{array}{rcl}
1 & a & c \\
0 & 1 & b \\
0 & 0 & 1
\end{array}
\right) \mbox{,}
$$
where $a, b, c$ are from $\Z$. We can consider also the group $G$ as a set of all integer triples endowed with the group law:
$$
(a_1,b_1,c_1)(a_2,b_2,c_2)= (a_1 + a_2, b_1 +b _2, c_1+ c_2 + a_1b_2)   \mbox{.}
$$
We note that $G$ may be presented as
$$<\eta, \gamma \; : \; [\eta,[ \eta, \gamma]]=1 =[\gamma, [\eta, \gamma]]    >$$
 with $\eta$ and $\gamma$
corresponding to the generators $(1,0,0)$ and $(0,1,0)$ (see, for example,~\cite[\S~2]{Ka}).
It is clear that the group $G$ is nilpotent of class $2$.

\begin{prop} \label{Heisen}
The group $\widetilde{K^*}/{\oo'}^*$ is isomorphic to the group $G= {\rm Heis}(3, \Z)$.
\end{prop}
\begin{proof}
We will construct this isomorphism explicitly.
 We consider a map
\begin{equation}  \label{map}
G \ni (a,b,c) \;  \longmapsto  \; (u^a t^b, q^{-c} \mu_{B, 0, b}) \in \widetilde{K^*}   \mbox{.}
\end{equation}

Using  the equality $B = tB$
we obtain an action of the element $t$ as $t(\mu_{B,r,s})=\mu_{B, r+1, s+1} $. Besides an action of the element $u$ is
$u(\mu_{B,r,s})= q^{r-s}\mu_{B,r,s}$ (we recall that the pair $t,u$  is the system of local parameters of $K$).
Now from an explicit description of the group $\widetilde{K^*}$ given above, by a direct calculation we obtain
that the composition of map~\eqref{map} with the quotient map $\widetilde{K^*} \to \widetilde{K^*}/{\oo'}^*$ is an isomorphism between groups $G$
and $\widetilde{K^*}/{\oo'}^*$. The proposition is proved.
\end{proof}

\subsection{Representations of the discrete Heisenberg group} \label{td2}
Let $K$ be a two-dimensional local field that is isomorphic to one of the fields in formula~\eqref{cases}.
We keep the notation from Section~\ref{tdl}.

The field $K$ is a discrete valuation field.
Its residue field $\bar{K} = \oo / \oo(1)$ is isomorphic to the following field (according to formula~\eqref{cases}):
$$
\df_q((u))  \mbox{,}  \qquad \qquad  L \mbox{,}  \qquad  \qquad \df_q((u))  \mbox{.}
$$

For any integers $s \le r$ the group $\oo(s) / \oo(r)$ is a locally compact totally disconnected group.
Let $\D(\oo(s)/\oo(r))$ be the space of $\mathbb{C}$-valued locally constant functions with compact support
on the group $\oo(s)/ \oo(r)$.
Let $\D'(\oo(s)/ \oo(r))$ be the space of distributions on the
group $\oo(s)/ \oo(r)$ (i.e. the dual vector space to the discrete space $\D(\oo(s)/\oo(r))$).

The following spaces which are two-dimensional analogs of the classical function and distribution spaces  were constructed in~\cite[\S~5.3]{OsipPar1}
and~\cite[\S~7]{OsipPar2}  using an idea from~\cite{K}:
\begin{eqnarray}
\label{for1}
\D_{\oo}(K) = \mathop{\lim_{\longleftarrow}}_s  \mathop{\lim_{\longleftarrow}}_{r > s} \D(\oo(s)/\oo(r)) \otimes_{\C}  \mu(\oo(r) \mid \oo)  \mbox{,} \\
\label{for2}
\D'_{\oo}(K) = \mathop{\lim_{\lrto}}_s  \mathop{\lim_{\lrto}}_{r > s} \D'(\oo(s)/\oo(r)) \otimes_{\C} \mu(\oo \mid \oo(r))  \mbox{,}
\end{eqnarray}
where transition maps in above limits are given by direct and inverse images of function and distribution spaces on locally compact groups.

There is a natural non-degenerate pairing:
$$
\D_{\oo}(K)  \times \D'_{\oo}(K)  \lrto \C  \mbox{.}
$$
The group $\widehat{K^*}$ acts naturally on the spaces $\D_{\oo}(K)$ and $\D'_{\oo}(K)$. We recall this action on $\D'_{\oo}(K)$ (the case of
$\D_{\oo}(K)$ is analogous).  We have to apply limits to the following action of an element $(g, \mu)$ from $\widehat{K^*}$:
\begin{eqnarray*}
\D'(\oo(s)/\oo(r)) \otimes_{\C} \mu(\oo \mid \oo(r)) \stackrel{g( \; )}{\lrto}  \D'(g\oo(s)/g\oo(r)) \otimes_{\C} \mu(g\oo \mid g\oo(r))  \\
  \D'(g\oo(s)/g\oo(r)) \otimes_{\C} \mu(g\oo \mid g\oo(r))
 \stackrel{\mu \otimes}{\lrto} \D'(g\oo(s)/g\oo(r)) \otimes_{\C} \mu(\oo \mid g\oo(r))  \mbox{.}
\end{eqnarray*}

We recall that the group ${\oo'}^*$ can be considered as a subgroup of the group $\widehat{K^*}$.
We are interested in the spaces of invariant elements with respect to the action of the group ${\oo'}^*$: $\D_{\oo}(K)^{{\oo'}^*}$
and $\D'_{\oo}(K)^{{\oo'}^*}$.

\begin{prop}
We have  $\D_{\oo}(K)^{{\oo'}^*} =0$.
\end{prop}
\begin{proof}
We have Lemma~\ref{lem1}. Besides,  the transition  maps in projective limits in formula~\eqref{for1} are ${\oo'}^*$-equivariant maps.   Therefore it is enough to prove that $\D(\oo(s)/ \oo(r))^{{\oo'}^*} =0$ for $r-s >1$.
Assume the converse. We take a non-zero element $f \in \D(\oo(s)/ \oo(r))^{{\oo'}^*}$. Since $f$ has a compact support,
 there is an integer $l $ such that
$f(x) =0$ when $x \notin u^l V$,  where $V=(B \cap \oo(s))/ (B \cap \oo(r))$. It contradicts to the fact $f(gy)=f(y)$ where $g$ is any from ${\oo'}^*$,
$y$ is any from $\oo(s)/\oo(r)$, because if $f(y) \ne 0$, then there is $g \in {\oo'}^*$ such that $gy \notin u^l V$,
 and hence $f(gy) \ne 0$. The proposition is proved.
\end{proof}

\bigskip

The group $G = \widetilde{K^*}/ {\oo'}^*$ acts naturally on the space $\D'_{\oo}(K)^{{\oo'}^*}$.

\bigskip

\noindent
{\bf Problem.} \, How to describe the space $\D'_{\oo}(K)^{{\oo'}^*}$ and the action of the group $G$ on it?
(See also~\cite[\S 5.4 (v)]{P2}.)

\bigskip

We will describe the following $G$-invariant infinite-dimensional subspace of the space $\D'_{\oo}(K)^{{\oo'}^*}$.

Let $\varphi \in \D'(\bar{K})^{\oo_{\bar{K}}^*} = \D'(\oo / \oo(1))^{\oo_{\bar{K}}^*}$.
For any integer $l$ we {\em define} an element
\begin{equation}  \label{vp}
\varphi_l = t^l(\vp)  \otimes \mu_{B, 0, l+1}  \;  \in  \;  \D'_{\oo}(K) \mbox{,}
\end{equation}
where an element $t^l(\vp) \in \D'(\oo(l)/ \oo(l+1))$ is an application of the action of $t^l$ to $\vp$,
an element $\mu_{B, 0, l+1} \in \mu(\oo \mid \oo(l+1))$  was defined in Section~\ref{tdl}.
We have that $\vp_l \in \D'_{\oo}(K)^{{\oo'}^*}$, because
$\vp \in \D'(\oo / \oo(1))^{\oo_{\bar{K}}^*}$ and, by lemma~\ref{lem1}, $g \mu_{B, 0, l+1} = \mu_{B, 0, l+1}$ for any $g \in {\oo'}^*$,
and the transition maps in inductive limits in formula~\eqref{for2} are ${\oo'}^*$-equivariant maps.

\begin{defin} \label{Psi}
We define a $\C$-vector subspace $\Psi$ of the vector space $\D'_{\oo}(K)^{{\oo'}^*}$  as the $\C$-linear span of  elements $\vp_l$ for all elements
$\vp \in \D'(\bar{K})^{\oo_{\bar{K}}^*}$
and all integers $l$.
\end{defin}

Using en explicit isomorphism from the proof of Proposition~\ref{Heisen} and a
description of the action of $\widehat{K^*}$ on $\D'_{\oo}(K)$,
we obtain the following proposition.

\begin{prop}  \label{exten}
The space $\Psi$ is a $G$-invariant subspace with an explicit action:
$$
(a,b,c) \cdot \vp_l = q^{-c-a(l+1)} a(\vp)_{l+b}   \mbox{,}
$$
where $(a,b,c) \in G$ and $a(\vp) \in \D'(\bar{K})^{\oo_{\bar{K}}^*}$ is an application of the action of the element $u^a \in \bar{K}^*$
to the element $\vp$.
\end{prop}
\begin{proof}
By direct calculations we have
\begin{multline*}
(a,b,c) \cdot \vp_l = \left(u^a t^b, q^{-c} \mu_{B, 0,b}\right)  \cdot \left(  t^l(\vp)  \otimes \mu_{B, 0, l+1} \right)= \\ =
\left( t^{l+b}(u^a(\vp)), q^{-c}\mu_{B, 0,b} \otimes u^a(\mu_{B, b , b+l+1} \right)= \\=
\left( t^{l+b}(a(\vp)), q^{-c-a(l+1)}   \mu_{B, 0,b} \otimes \mu_{B, b , b+l+1}               \right) =
 q^{-c-a(l+1)} a(\vp)_{l+b}  \mbox{.}
\end{multline*}
The proposition is proved.

\end{proof}

\medskip

We recall the notion of the induced representation for discrete groups.
\begin{defin} \label{ind}
Let  $G_1 \subset G_2$ be groups,
and $\tau: G_1 \to \Aut_{\C} V$ be a representation of the group $G_1$ in a $\C$-vector space $V$.
A space of the induced representation ${\rm ind }^{G_2}_{G_1}(\tau)$ consists of all functions $f: G_2  \to V $ subject
to the following conditions:
\begin{itemize}
\item[1)]  $f(g_2g_1)= \tau(g_1^{-1}) f(g_2)$
for any elements $g_1 \in G_1$, $g_2 \in G_2$,
\item[2)] the support ${\rm Supp} (f)$ is contained in finite number of left cosets $G_2/G_1$.
\end{itemize}
An action of the group $G_2$ is the following: for any $f \in {\rm ind }^{G_2}_{G_1}(\tau)$ and $g \in G_2$ we have $(g \cdot f)(x)=f(g^{-1}x)$
where $x \in G_2$.
\end{defin}

We consider  subgroups $H =\Z $, $C =\Z$ and $P =\Z$ of the group $G$   given as
$$
H \ni a \mapsto (a,0,0) \in G    \mbox{,}  \qquad P \ni b \mapsto (0,b,0) \in G  \mbox{,}
\qquad
C \ni c \mapsto (0,0,c) \mbox{,} \qquad $$
We consider a representation
\begin{equation}  \label{sigma}
\sigma: H \oplus C \lrto \Aut\nolimits_{\C} (\D'(\bar{K})^{\oo_{\bar{K}}^*}) \mbox{,}  \qquad   \sigma(a \oplus c)(\vp)= q^{-c-a} a(\vp) \mbox{,}
\end{equation}
where $\vp \in \D'(\bar{K})^{\oo_{\bar{K}}^*} $ and $a(\vp) = u^a(\vp)  \in \D'(\bar{K})^{\oo_{\bar{K}}^*}$.

 The group $H \oplus C$ is a subgroup of the group $G$ via the map $a \oplus c  \mapsto (a, 0,c)$.
We note that the set of left cosets $G / (H \oplus C)$  can be  naturally identified with  the set of elements of the subgroup $P$.
By Definition~\ref{ind}, the space ${\rm ind}^G_{H \oplus C} (\sigma)$ can be identified with the space of functions with finite support from
$P$ to $\D'(\bar{K})^{\oo_{\bar{K}}^*}$. Therefore the space  ${\rm ind}^G_{H \oplus C} (\sigma)$ is linearly generated by elements $\Delta(m,\vp)$
where $m \in \Z$, $\vp \in  \D'(\bar{K})^{\oo_{\bar{K}}^*}$. Here $\Delta(m,\vp)$ is a function from $P = \Z$ to $\D'(\bar{K})^{\oo_{\bar{K}}^*}$
such that
\begin{itemize}
\item $\Delta(m,\vp)(n) = \vp$  if $m=n$,
\item  $\Delta(m,\vp)(n) = 0$ if $m \ne 0$.
\end{itemize}

We consider {\em a map} $\beta$ from the space ${\rm ind}^G_{H \oplus C} (\sigma)$ to the space $\Psi$ given as
$$\beta(\Delta(m,\vp))= \vp_m  \mbox{.}$$

\begin{Th}  \label{t1}
The map $\beta$ induces an exact sequence of $G$-representations
$$
0 \lrto \Ker \beta \lrto  {\rm ind}^G_{H \oplus C} (\sigma)   \stackrel{\beta}{\lrto} \Psi \lrto 0  \mbox{.}
$$
The subspace  $\Ker \beta$ has a basis which consists of elements
$$v_m = \Delta(m, \delta_0) - \Delta(m+1, \mu_{B,1,0}) \qquad (m \in \Z)$$
where $\delta_0$ is the delta-function at $0$, and $\mu_{B,1,0}  \in \mu(\bar{K}) \subset  \D'(\bar{K})^{\oo_{\bar{K}}^*} $. Moreover,
the space $\Ker \beta$
is an irreducible $G$-representation  which is isomorphic to $G$-representation
${\rm ind}^G_{H \oplus C} (\chi_0)  $ with the character $\chi_0: H \oplus C  \lrto \C^*$ given as \linebreak $\chi_0(a \oplus c)= q^{-c}$.
\end{Th}
\begin{proof}
We check that the map $\beta$ is a map of $G$-representations.
Let $g =(a,b,c) \in G$. We have $g^{-1}= (-a, -b, ab-c)$. Using the first condition from definition~\ref{ind} we extend a function
$\Delta(m,\vp)$ to a function from $G$ to  $\D'(\bar{K})^{\oo_{\bar{K}}^*}$.
We have for any $g =(a,b,c)$ and any integer $k$
\begin{multline} \label{ind-form}
(g \cdot \Delta(m,\vp))((0,k,0)) = \Delta(m,\vp))(g^{-1} \cdot (0,k,0))= \\
= \Delta(m,\vp)( (-a, -b, ab-c) (0,k,0))= \Delta(m,\vp) ((-a, -b+k, ab-ak-c)) = \\
= \Delta(m,\vp) ((0,-b+k, 0) (-a, 0, ab-ak-c)) = \\ = \sigma\big(a \oplus (c+ak-ab)\big) \big( \Delta(m,\vp) ((0,-b+k, 0)) \big)= \\
= q^{-c-a(k-b+1)} a\big( \Delta(m,\vp) ((0,-b+k, 0) ) \big) \mbox{.}
\end{multline}
Now if $k =b+m$ then the last formula is equal to $q^{-c-a(m+1)} a(\vp)$. If $k \ne b+m$, then the last formula is equal to $0$.
Hence we have
\begin{equation}  \label{ind-form-2}
(a,b,c) \cdot \Delta(m,\vp) = q^{-c-a(m+1)} \Delta(b+m, a(\vp))  \mbox{.}
\end{equation}
Therefore using the formula from Proposition~\ref{exten} we obtain that the map $\beta$ commutes with the action of $G$.

The surjectivity  of the map $\beta$ follows from the definition of this map. We calculate the kernel of $\beta$. We have from the construction that elements $v_m  $ ($m \in \Z$) belong to the space $\Ker \beta$. It is also clear that elements $v_m$  ($m \in \Z$) are linearly independent over $\C$. Let an element $\sum\limits_{i=1}^n \Delta(m_i, \vp_i)   $ belongs to $\Ker \beta$, where
$ m_i < m_{i+1} $ for $ 1 \le i \le (n-1)$, and $\vp_i \in \D'(\bar{K})^{\oo_{\bar{K}}^*}$ for $1 \le i \le n$. Hence we obtain
\begin{equation}  \label{equat}
\sum\limits_{i=1}^n (\vp_i)_{m_i}=0
\end{equation}
in the space $\Psi$. Moreover,  equality~\eqref{equat} is satisfied  in the space
$$\D'(\oo(m_1)/ \oo(m_n +1)) \otimes_{\C} \mu(\oo(m_n+1) \mid \oo) \mbox{,}$$
 since the transition maps in limits in formula~\eqref{for2} are injective.
We fix $\mu_{B, m_n+1, 0}  \in \mu(\oo(m_n+1) \mid \oo)$. After that we can consider an analog of equality~\eqref{equat}
in $\D'(\oo(m_1)/ \oo(m_n +1))$. In the last equality we consider the supports of summands, where every summand is an   element
of $\D'(\oo(m_1)/ \oo(m_n +1))$ (see, for example,  the definition of the support of distribution in~\cite[Def.~4]{OsipPar1}).
From this equality we obtain  that ${\rm supp} (\vp_1)=\{0\}$ as a subset in $\bar{K} = \oo/\oo(1)$. Therefore $1(\vp_1)=\vp_1$, and from Example~\ref{example1} we
 obtain that $(\vp_1)_{m_1}= e \cdot (\delta_0)_{m_1}$, where $e \in \C$.
Now we use $(\delta_0)_{m_1} = (\mu_{B,1,0})_{m_1+1}$. From induction on the length $m_{n+1} -m_1$ we obtain that the space $\Ker \beta$
is linearly generated by elements $v_m$ with $m \in \Z$.

From formula~\eqref{ind-form-2} we have $(a,b,c) \cdot v_m= q^{-c-a(m+1)}v_{m+b}$  for any ${(a,b,c) \in G}$ and $m \in \Z$.  Using calculations
analogous to~\eqref{ind-form}, it is easy to see that
this action corresponds to the action of $G$ on
${\rm ind}^G_{H \oplus C} (\chi)$ (if we consider elements $v_m$ as functions $\delta_m$   from the subgroup $P$ to $\C$ supported on $m \in \Z$), where the character $\chi$ is given as $\chi(a\oplus c)= q^{-c-a}$. If we shift the basis: $\tilde{v}_m= v_{m-1}$, then the action is
\begin{equation}  \label{act_v_t}
(a,b,c) \cdot \tilde{v}_m =  q^{-c-am} \tilde{v}_{m+b}  \mbox{.}
\end{equation}
The last action corresponds to the action of $G$ on ${\rm ind}^G_{H \oplus C} (\chi_0)$,  where the character $\chi_0$ is given as $\chi_0(a\oplus c)= q^{-c}$. Thus, $G$-representations ${\rm ind}^G_{H \oplus C} (\chi)$ and ${\rm ind}^G_{H \oplus C} (\chi_0)$
are isomorphic.

We explain the irreducibility of $G$-representation ${\rm ind}^G_{H \oplus C} (\chi_0)$ (see another argument in~\cite[Prop.~1]{P1}).
Consider any nonzero element
$$x = \sum_{i=1}^n \alpha_{i} \tilde{v}_{m_i} \; \in  \; \Ker \beta  \mbox{,}$$ where
 $m_i \in \Z$
 and
$\alpha_{i} \in \C^*$
 for $1 \le i \le n$, besides $m_k < m_{k+1}$ for $1 \le k \le n-1$. We will prove that $\C[G] \cdot x = \Ker \beta$.
 Consider a column $Y= [y_1; \ldots ; y_n]$
 whose entries $y_j =(-j+1, 0,0) \cdot x$ with $1 \le j \le n$ are elements of the space $\Ker \beta$. From formula~\eqref{act_v_t}
 we obtain that $Y = A Z$, where the column $Z= [z_1; \ldots ; z_n]$ has entries $z_i = \alpha_i \tilde{v}_{m_i}$, where $1 \le i \le n$,  and the entries of the square matrix $A$
 are as follows: $A_{j,i}= q^{(j-1)m_i}$, where $1 \le j \le n$ and $1 \le i \le n$. The matrix $A$ is the Vandermonde matrix with non-zero determinant equal to
 $$\prod_{1 \le j_1 < j_2 \le n} (q^{m_{j_2}} - q^{m_{j_1}}) \mbox{.}$$
 Therefore there is the inverse matrix $A^{-1}$
 with entries from the field $\Q$. Hence we obtain $Z = A^{-1} A Z = A^{-1} Y$. Therefore any element $z_i = \alpha_i \tilde{v}_{m_i}$, where $1 \le i \le n$, is a linear combination (with coefficients from $\Q$) of elements  $y_j =(-j+1, 0,0) \cdot x$, where $1 \le j \le n$.
 Hence any element $\tilde{v}_{m_i}$, where $1 \le i \le n$,  belongs to the space $\C[G] \cdot x$. By formula~\eqref{act_v_t}, the action by means of  an element
 $(0,b,0)$, where $b \in \Z$, maps an element $\tilde{v}_k$  to the element $\tilde{v}_{k+b}$, where $k \in \Z$. Therefore, we obtain that  any element $\tilde{v}_i$ with $i \in \mathbb{Z}$ belongs to the space $\C[G] \cdot x$.
 The theorem is proved.
 \end{proof}

\bigskip
We consider subrepresentations $\sigma_1 \subset \sigma_2$ of the representation $\sigma$ (see formula~\eqref{sigma} and formula~\eqref{form1}):
$$
\sigma_1: H \oplus C \lrto \Aut\nolimits_{\C} ( \D(\bar{K})^{\oo_{\bar{K}}^*} \otimes_{\C} \mu(\bar{K})   )   \mbox{,}
$$
$$
\sigma_2 : H \oplus C \lrto \Aut\nolimits_{\C} ( \E(\bar{K})^{\oo_{\bar{K}}^*} \otimes_{\C} \mu(\bar{K})   ) \mbox{.}
$$
We have the following embeddings of representations of the group $G$:
$$
{\rm ind}^G_{H \oplus C} (\sigma_1)  \; \subset \;  {\rm ind}^G_{H \oplus C} (\sigma_2) \; \subset \; {\rm ind}^G_{H \oplus C} (\sigma)  \mbox{.}
$$

\begin{prop}
The map $\beta$ restricted to ${\rm ind}^G_{H \oplus C} (\sigma_2)$ is an isomorphism with its image in $\Psi$.
\end{prop}
\begin{proof}
From theorem~\ref{t1} we have an explicit description of elements $v_m$ ($m \in \Z$) which form a basis in the space $\Ker \beta$.
Since the element $\delta_0  $  from the space $\D'(\bar{K})$ does not belong to the subspace $\E(\bar{K}) \otimes_{\C } \mu(\bar{K})$,
we obtain that any non-zero element $\sum_{i=1}^n a_i v_i  $ ($a_i \in \C$) does not belong to   ${\rm ind}^G_{H \oplus C} (\sigma_2)$. The proposition is proved.
\end{proof}

\medskip

Let $\alpha \in \C^*$. We recall that the $1$-dimensional $\C$-vector subspace $\Phi_{\alpha}$ of the space $\D'(\bar{K})^{\oo_{\bar{K}}^*} $  was introduced in Proposition~\ref{alpha}.
We consider a subrepresentation $\sigma_{\alpha}$ of the representation $\sigma$ (see formula~\eqref{sigma}):
$$
\sigma_{\alpha}: H \oplus C \lrto \Aut\nolimits_{\C} ( \Phi_{\alpha}   ) =\C^*  \mbox{,} \qquad
\sigma_{\alpha} (a \oplus c)  \longmapsto q^{-c-a} \alpha^a  \mbox{.}
$$
We define a $G$-subrepresentation  $\Psi_{\alpha}= {\rm ind}^G_{H \oplus C} (\sigma_{\alpha}) $ of $G$-representation ${\rm ind}^G_{H \oplus C} (\sigma) $.
We consider the element $g_{\alpha} \in \Phi_{\alpha}$ (see Remark~\ref{nt1}). Now elements $\Delta(m, g_{\alpha})$ ($m \in \Z$) form a basis of the subspace $\Psi_{\alpha}$
inside the space  ${\rm ind}^G_{H \oplus C} (\sigma)$. From formula~\eqref{ind-form-2}  we have the following action of the group $G$ on these elements:
\begin{equation}  \label{ind-form-3}
(a,b,c) \cdot \Delta(m, g_{\alpha}) = q^{-c-a(m+1)} \alpha^a \Delta(b+m, g_{\alpha})  \mbox{.}
\end{equation}
\begin{Th}  \label{main}
$G$-representations $\Psi_{\alpha}$ have the following properties.
\begin{enumerate}
\item \label{it1} For any $\alpha \in \C^*$ the space $\Psi_{\alpha}$ is an irreducible $G$-representation, and the map $\beta$ restricted to
$\Psi_{\alpha}$ is an isomorphism with its image in $\Psi$.
\item \label{it2} If $\alpha =q$, then $\Psi_{\alpha }  \subset  {\rm ind}^G_{H \oplus C} (\sigma_2) $. If $\alpha \ne q$,
then $\Psi_{\alpha}  \cap \, {\rm ind}^G_{H \oplus C} (\sigma_2) =\{0\}$ inside ${\rm ind}^G_{H \oplus C} (\sigma) $.
\item \label{it3} We have $\beta(\Psi_1)= \beta(\Psi_q)$ as subspaces in the space $\Psi$. For any pair $(\alpha_1, \alpha_2) \in \C^* \times \C^*$ such that $\alpha_1 \ne \alpha_2$,
$(\alpha_1, \alpha_2) \ne (q,1)$ and $(\alpha_1, \alpha_2) \ne (1,q) $ we have
$\beta(\Psi_{\alpha_1}) \cap  \beta(\Psi_{\alpha_2}) =\{0\}$
 inside the space $\Psi$.
 \item We have embeddings of $G$-representations:
 \begin{equation}  \label{dirsum}
 \bigoplus_{\alpha \in \C^*} \Psi_{\alpha}   \; \subset \; {\rm ind}^G_{H \oplus C} (\sigma) \qquad \mbox{and} \qquad
 \mathop{\bigoplus_{\alpha \in \C^*}}_{\alpha \ne 1} \beta(\Psi_{\alpha})   \; \subset \; \Psi  \mbox{.}
 \end{equation}
\item \label{it4} For any $\alpha_1$ and $\alpha_2$ from $\C^*$ we have that the $G$-representations $\Psi_{\alpha_1}$ and $\Psi_{\alpha_2}$
are isomorphic if and only if $\alpha_1 = q^l \alpha_2$ for some integer $l$. Thus the $G$-representation $\Psi$ contains  irreducible
$G $-representations which are parametrized by the elliptic curve $\C^* / q^{\Z}$. Any such irreducible $G$-representation is included with at least countable multiplicity in the $G$-representation $\Psi$.
\end{enumerate}
\end{Th}
\begin{proof}
\begin{enumerate}
\item  The space $\Psi_{\alpha}$ is an irreducible $G$-representation by
an argument which is similar to the end of the proof of Theorem~\ref{t1} (see  another argument in~\cite[Prop.~1]{P1}). We prove that $\Ker \beta \cap \Psi_{\alpha} =0$.
By Theorem~\ref{t1} the $G$-representation $\Ker \beta$ is irreducible. Therefore it is enough to prove that an element $v_m$ from $\Ker \beta$
does not belong to $\Psi_{\alpha}$. The latter follows from the fact that $\delta_0 \in \Phi_0$  and $\mu_{B,1,0} \in \Phi_q$
(see Example~\ref{example1}).
\item We have $\Psi_q  \subset {\rm ind}^G_{H \oplus C} (\sigma_2)$, since the space $\Phi_q$ is equal to the space $\mu(\bar{K})$ of $\C$-valued Haar measures on $\bar{K}$ (see Example~\ref{example1}). The subspace of all constant functions on $\bar{K}$ is the set of all eigenvectors in $\E(\bar{K})^{\oo_{\bar{K}^*}}$ with respect to the action of local parameter $u \in \bar{K}^*$ (or $1 \in \Z$). Therefore if $\alpha \ne q$,
    then ${\Phi_{\alpha }  \cap (\E(\bar{K})^{\oo_{\bar{K}}^*} \otimes_{\C} \mu(\bar{K}) ) =\{0\}}$. Hence if $\alpha \ne q$, then
    ${\Psi_{\alpha}  \cap \, {\rm ind}^G_{H \oplus C} (\sigma_2) =\{0\}}$.
 \item  This item follows from explicit description of basis vectors $v_m$ ($m \in \Z$) of the space $\Ker \beta$ (see Theorem~\ref{t1}),
 from description of basis vectors of the space $\Psi_{\alpha}$,
 and from assertion~\ref{it1}.
 \item Elements $\Delta(m, g_{\alpha})$ ($m \in \Z$, $\alpha \in \C^*$) from ${\rm ind}^G_{H \oplus C} (\sigma)$
 are linearly independent over $\C$. (Indeed, it is enough to see that elements $g_{\alpha}$ with different $\alpha \in \C^*$ are linearly independent in $\D'(\bar{K})$. This is true, because these elements are eigenvectors with different eigenvalues of linear operator induced by multiplication action of local parameter $u \in \bar{K}^*$.) These elements form a basis of the space
 $\bigoplus\limits_{\alpha \in \C^*} \Psi_{\alpha}$.
 Thus we obtained the first formula in~\eqref{dirsum}. We note that
 $$
 \Ker \beta \; \subset \; (\Psi_1 \oplus \Psi_q)  \qquad \mbox{and}  \qquad  (\Psi_1 \oplus \Psi_q)/ \Ker \beta \;  = \; \beta(\Psi_q) \mbox{.}
 $$
 Hence and from the first formula in~\eqref{dirsum} we obtain the second formula in~\eqref{dirsum}.
 \item \label{last-item}
 For this item we use the previous items of the theorem. Now we prove that $G$-representations $\Psi_{\alpha_1}$ and $\Psi_{\alpha_2}$
are isomorphic if and only if $\alpha_1 = q^l \alpha_2$ for some integer $l$ (see also~\cite[Prop.~1]{P1} and~\cite[Th.~5]{AP}).
If $\alpha_1 = q^l \alpha_2$, then this isomorphism is constructed by a shift of  the elements of the basis:
${\Delta(m, g_{\alpha_1}) \mapsto \Delta(m-l, g_{\alpha_2})}$ (see formula~\eqref{ind-form-3}). Suppose that above two representations are isomorphic. We note that elements of the one-dimensional vector spaces $\C \cdot \Delta(m, g_{\alpha_i})$ are the only eigenvectors of elements of the subgroup $H$ under the action on $\Psi_{\alpha_i}$. (In other words,   we have the weight decomposition of $\Psi_{\alpha_i}$ in above one-dimensional vector spaces with respect to the subgroup $H$.) Therefore the isomorphism of $G$-representations maps the element  $\Delta(0, g_{\alpha_1})$ to the element  $d \Delta(-l, g_{\alpha_2})$
for certain integer $l$ and $d \in \C^*$. The action of the element $(1,0,0)$ from $G$ on above two elements by formula~\eqref{ind-form-3} gives $\alpha_1 = q^l \alpha_2$.

The theorem is proved.
 \end{enumerate}
\end{proof}

\section{Fourier transform}

Let $L$ be a one-dimensional local field  with a finite residue field $\df_q$ (see Section~\ref{sec2}). Let $u$ be a local parameter in $L$. Then the additive group of the field $L$ is
 identified with its dual Pontryagin group via the pairing:
$$
\chi \, : \, L \times L \lrto \mathbb{T} \mbox{,} \qquad  \qquad  \chi(l_1, l_2)  = \exp(2 \pi i \cdot \lambda_L (l_1 l_2))  \mbox{,}
$$
where $\mathbb{T}$ is the circle group, i.e. the unit circle in $\C$, and the map $\lambda_L : L \to \dq/\Z$ is defined as following.  If $L \simeq \df_q((u))$, then for an element $l=\sum a_i u^i  $ from $L$ (where $a_{i} \in \df_q$) we define $\lambda_L (l)= \tr_{\df_q/ \df_p} (a_{-1}) \in p^{-1}\Z/\Z$. If $L$ is a finite extension of $\dq_p$, then the map $\lambda_L$ is the composition of maps:
$$
L \stackrel{\tr_{L / \dq_p}}{\lrto} \dq_p  \lrto \dq_p/\Z_p  \lrto \dq/\Z  \mbox{.}
$$

For any $\mu \in \mu(L) \subset \D'(L)$, $\mu \ne 0$ we have the Fourier transform
$$\F_{\mu} \; : \; \D(L) \to \D(L) \mbox{,} \qquad \qquad (\F_{\mu}(f))(y)= \mu(f \cdot \chi_{-y}) )   \mbox{,}
$$
where $f \in \D(L)$, $y \in L$, and $\chi_{-y}$ is a locally constant function on $L$ defined as $\chi_{-y}(x)= \chi(-y,x)$ for $x \in L$.
As the dual map to the above map we have the Fourier transform  on the distribution space
$$ \F_{\mu^{-1}} \; :  \; \D'(L)  \to \D'(L)  \mbox{.}$$

By direct calculations, for any $l \in L^*$ and $f \in \D(L)$ we have
$$\F_{\mu} (l(f))= (\mu/l(\mu)) \cdot l^{-1}(\F_{\mu}(f))  \mbox{,}$$
where $(\mu/ l(\mu)) \in \C^*$. Hence for any $G \in \D'(L)$ we obtain
$$\F_{\mu^{-1}}(l(G))= (l(\mu)/\mu)  \cdot l^{-1}(\F_{\mu^{-1}}(G)) \mbox{.}$$
Hence we have  well-defined maps
$$\F_{\mu} : \D(L)^{\oo_L^*} \lrto  \D(L)^{\oo_L^*}  \mbox{,} \qquad  \F_{\mu^{-1}} : \D'(L)^{\oo_L^*}  \lrto \D'(L)^{\oo_L^*}  $$
and the following proposition about spaces $\Phi_{\alpha}$ (recall also Proposition~\ref{alpha}).
\begin{prop} \label{prop7}
For any $\alpha \in \C^*$ we have $\F_{\mu^{-1}} (\Phi_{\alpha}) = \Phi_{q \alpha^{-1}}$.
\end{prop}
\begin{nt} \label{nt2} {\em
We recall that in Remark~\ref{nt1} we have constructed elements $g_{\alpha} \in \Phi_{\alpha}$ for any $\alpha \in \C^*$.
Let $\mu_0$ is a Haar measure on $L$ such that its appication to the open compact subgroup $\oo_L$ is equal to $1$.
We have $\F_{\mu_0}(g_{\alpha})= g_{q \alpha^{-1}}$.
}
\end{nt}

\bigskip

Let $K$ be a two-dimensional local field that is isomorphic to one of the fields in formula~\eqref{cases}.
We keep the notation from Sections~\ref{tdl} and~\ref{td2}. The additive group of the field $K$ is self-dual via the pairing:
$$
K  \times K  \lrto \mathbb{T} \mbox{,} \qquad  \qquad k_1 \times k_2  \longmapsto \exp(2 \pi i \cdot \Lambda(k_1 \cdot k_2)) \mbox{,}
$$
where the map $\Lambda : K \lrto \dq/\Z$
is defined as following. If $K$ is the first field from formula~\eqref{cases}, then for an element $k=\sum a_{i,j} u^i t^j$ from $K$
(where $a_{i,j} \in \df_q $)
we define $\Lambda(k)= \tr_{\df_q/ \df_p} (a_{-1,-1}) \in p^{-1}\Z /\Z$.
If $K$ is the second field from formula~\eqref{cases}, then for an element $k = \sum a_i t^i$ from $K$ (where $a_i \in L$ )
we define $\Lambda (k) = \lambda_L (a_{-1}) \in \dq/\Z$. If $K$ is the third field from formula~\eqref{cases}, then
an element $k$
from $K$ can be presented as infinite in both sides sum $\sum a_i u^i$ (where $a_i \in M$), and, by definition,
$\Lambda(k)= \lambda_M(a_{-1}) \in \dq/\Z$.

The two-dimensional Fourier transform
\begin{equation}  \label{ft}
\F  \;  : \; \D'_{\oo}(K)  \lrto \D'_{\oo}(K)
\end{equation}
was defined in~\cite[\S~5.4]{OsipPar1}  and in~\cite[\S~8]{OsipPar2}. (This Fourier transform is obtained as an application of double inductive limits from formula~\eqref{for2}
to one-dimensional (usual) Fourier transform applied to the spaces $\D'(\oo(s)/\oo(r))$  for all pairs of integers $r < s$.)
Note that for the two-dimensional Fourier transform from formula~\eqref{ft} it is important the self-duality of the additive group of the field $K$
which is described above.

From~\cite[Prop.~25]{OsipPar1} and~\cite[Prop.~16]{OsipPar2} we obtain (after restriction to the space $\D'_{\oo}(K )^{{\oo'}^*}$) that the map $\F$ is well-defined as a map
$$
\F \; : \; \D'_{\oo}(K )^{{\oo'}^*}  \lrto \D'_{\oo}(K)^{{\oo'}^*}   \mbox{.}
$$
In what follows we will consider only the last map which we denote by the same letter $\F$ as  original map~\eqref{ft}.
 We obtain the following properties from~\cite[Prop.~24,~25]{OsipPar1} and~\cite[{Prop.~15,~16}]{OsipPar2}:
 $$
 \F^2 = {\rm id}  \qquad \mbox{and}  \qquad
(a,b,c) \circ \F = \F \circ (-a,-b,c) \mbox{,}
 $$
where $(a,b,c) \in G$. We note that the map $(a,b,c) \mapsto (-a,-b,c)$ is a second order automorphism of the group $G$.

Let $\varphi \in \D'(\bar{K})^{\oo_{\bar{K}^*}} $.
For any integer $l$ we have the certain  element
$\varphi_l$ (see formula~\eqref{vp}). By direct calculations we have
\begin{equation} \label{exfor}
\F (\vp_l)= (\F_{\mu_0}(\vp))_{-1-l} \mbox{,}
\end{equation}
where  $\mu_0 \in \mu(\bar{K})$ was defined in Remark~\ref{nt2}. Hence we obtain that $\F(\Psi)=\Psi$.
\begin{prop}
We have the following properties.
\begin{enumerate}
\item $\F(\beta(\Psi_1))= \F(\beta(\Psi_q))= \beta(\psi_1) =\beta(\psi_q)$.
\item $\F (\beta(\Psi_{\alpha})) = \beta(\Psi_{q \alpha^{-1}})$ for any $\alpha \in \C^*$.
\end{enumerate}
\end{prop}
\begin{proof} It follows from  formula~\eqref{exfor} and Proposition~\ref{prop7}. The proposition is proved.
\end{proof}

\begin{nt}{\em
Using Remark~\ref{nt2} and formula~\eqref{exfor} we have $\F ((g_{\alpha})_l)= (g_{q \alpha^{-1}})_{-1-l}$
for $l \in \Z$ and $\alpha \in \C^*$. Inside the space $\beta(\Psi_1)$ we have $\F( (\delta_0)_k ) = (\delta_0)_{-k-2}$ for an integer $k$.
Using that $(\delta_0)_k = (\mu_0)_{k+1}$ we have $\F((\mu_0)_k)= (\mu_0)_{-k}$.
}
\end{nt}

 \section{Traces and ``loop rotations''} \label{Traces}

 We recall that in the proof of item~\ref{last-item} of Theorem~\ref{main}  we have constructed a distinguished basis of the  $G$-representation space $\Psi_{\alpha}$ (where $\alpha \in \C^*$) such that
 the basis consists of common eigenvectors for all elements  $h$ of the subgroup $H$  and if $h \ne 0$
 then distinct elements of the basis
correspond to distinct eigenvalues of $h$.
 Note that this property defines the elements of this basis
 up to the multiplication on elements of $\C^*$.  Elements $\Delta(l, g_{\alpha})$ (where $l \in \Z$) which we considered above form this basis.
 Therefore the traces (when they exist) of elements of $G$ with respect to the basis given by elements $\Delta(l, g_{\alpha})$ (where $l \in \Z$)
 are  invariants of $G$-representation $\Psi_{\alpha}$.

Let $(a,b,c) \in G$. From formula~\eqref{ind-form-3} we obtain the following expressions for traces $\tr_{\alpha}$ in the $G$-representation $\Psi_{\alpha}$:
\begin{itemize}
\item $\tr_{\alpha}((a,b,c))= 0$ if and only if $b \ne 0$;
\item $\tr_{\alpha}(a,0,c) = q^{-c-a} \alpha^{a} \sum\limits_{l \in \Z} q^{-al}$.
\end{itemize}
Clearly, the last expression is a divergent series.

To obtain  convergent series as traces we will extend the group $G$ to the group $\tilde{G}= G \rtimes  \Z$ which will naturally acts on the space $\Psi$ as well. Moreover,  the subspace  $\beta(\Psi_{\alpha}) \subset \Psi$ for any $\alpha \in \C^*$ defines a subrepresentation of the group $\tilde{G}$.  Hence we will obtain an action of the group $\tilde{G}$ on the isomorphic space $\Psi_{\alpha}$ (recall Theorem~\ref{main}).
To construct the group $\tilde{G}$
 we will consider automorphisms of the central extension
 $\widetilde{K^*}$ that are
   analogs of  automorphisms of ``loop rotation'' from the theory of loop groups.
We note that an extended discrete Heisenberg group was also considered in~\cite{P1} to obtain the traces which are  convergent series.
 But the methods of~\cite{P1} were other than consideration of analogs of ``loop rotations'' and action on the space $\Psi$, which is a subspace of the space of distributions on a two-dimensional local field.

\medskip

In what follows we will consider only a two-dimensional local field $K$ that is isomorphic either to the first or to the second field from formula~\eqref{cases}.
Let $t, u$ be a system of local parameters of $K$ as it is described in the beginning of Section~\ref{tdl}.
For any $d \in \Z$ we consider an automorphism $R_d$ of the field $K$ given as:
\begin{itemize}
\item $R_d \big(\sum\limits_{i,j} a_{i,j} u^i t^j\big)  = \sum\limits_{i,j} a_{i,j} u^{i + jd} t^j $, where $a_{i,j} \in \df_q$ and $K \simeq \df_q((u))((t))$;
\item $R_d \big(\sum\limits_j  b_j t^j \big)  = \sum\limits_j b_j u^{jd} t^j $, where $b_j  \in L$ and $K \simeq L((t))$ with $L$ which is a finite extension of $\Q_p$.
\end{itemize}
It is clear that we have a homomorphism $\Z \to \mathop{\rm Aut} (K)$ given as $d \mapsto R_d$.

We note that $R_d (\oo(i)) = \oo(i)$ for any integer $i$. Hence we have a well-defined continuous automorphism $R_d$ of the locally compact group
$\oo(i)/ \oo(j)$ for any integers $i \le j$. This automorphism is the identity morphism on the group $\oo(0)/ \oo(1)$.
Thus, we obtain an action of the group $\Z$ on the $\C$-vector spaces $\D(\oo(i)/ \oo(j))$ and $\D'( \oo(i))/\oo(j) $ via the maps induced by the maps $R_d$ (which we will denote by the same letters $R_d$).

Since the automorphism $R_d$ maps an element from $\D'( \oo(i))/\oo(j)$ that is invariant under all translations from $\oo(i)/\oo(j)$
 to an element  that is invariant under all translations from $\oo(i)/\oo(j)$, we obtain that $R_d $ induces an automorphism of the one-dimensional $\C$-vector space
 ${\mu(\oo(i)/ \oo(j)) = \D'(\oo(i)/ \oo(j))^{\oo(i)/ \oo(j)}}$.
 Therefore we have a well-defined action of the group $\Z$ on one-dimensional $\C$-vector spaces $\mu(\oo(r)  \mid \oo(s))$ (for any integers $r$ and $s$) and on the $\C$-vector spaces $\D_{\oo} (K)$ and $\D'_{\oo} (K)$ via the maps induced by the maps $R_d$ (which we denote by the same letters).

 \begin{lemma} \label{compute}
 For any integers $b$ and $d$ we have
 $$
 R_d(\mu_{B, 0,b}) = q^{-\frac{b(b-1)d}{2}}  \mu_{B,0,b} \mbox{.}
 $$
 \end{lemma}
\begin{proof}
We consider the case $b \ge  0$. The case $b < 0$ follows from analogous computations. We have
\begin{multline*}
R_d(\mu_{B, 0,1}  \otimes \ldots \otimes \mu_{B, b-1,b})
= R_d(\mu_{B, 0,1})  \otimes \ldots \otimes R_d (\mu_{B, b-1,b})=\\
= \mu_{B, 0,1}  \otimes q^{-d} \mu_{B, 1,2}  \otimes  \ldots \otimes  q^{-d(b-2)} \mu_{B,b-2,b-1}  \otimes  q^{-d(b-1)} \mu_{B, b-1,b} = \\
= q^{-\frac{b(b-1)d}{2}} \mu_{B, 0,1}  \otimes \ldots \otimes \mu_{B, b-1,b} = q^{-\frac{b(b-1)d}{2}}  \mu_{B,0,b}  \mbox{.}
\end{multline*}
The lemma is proved.
\end{proof}

From the construction of the central extension $\widehat{K^*}$ we have that the group $\Z$ acts on $\widehat{K^*}$ via the group automorphisms induced by the maps $R_d$ (which we denote by the same letters). From Lemma~\ref{compute} we have that this action is also an action on the central extension
$\widetilde{K^*}$. Clearly, $R_d ({\oo'}^*)= {\oo'}^*$ for any $d \in \Z$.
Hence the group $\Z$ acts  on the group $G \simeq \widetilde{K^*}/ {\oo'}^*$  via the group automorphisms induced by the maps $R_d$ (which we denote by the same letters).  Using an explicit description of the
isomorphism $G \simeq \widetilde{K^*}/ {\oo'}^*$  given in the proof of Proposition~\ref{Heisen}  we calculate now this action, i.e. the homomorphism $\Z
\to \mathop{\rm Aut}(G)$.
We have (using also Lemma~\ref{compute}):
\begin{multline}  \label{expl-auto}
R_d ((a,b,c))= R_d\left((u^at^b, q^{-c} \mu_{B, 0,b})\right)= \left(R_d(u^{a} t^b), q^{-c}R_d(\mu_{B, 0,b})\right) = \\ = \left(u^{a+db} t^b, q^{-c-\frac{b(b-1)d}{2} } \mu_{B, 0,b})\right) = \left(a +db, b, c + \frac{b(b-1)d}{2} \right)  \mbox{.}
\end{multline}

\begin{nt} \em
We can consider also formula~\eqref{expl-auto} as a formula which gives a homomorphism from the subgroup  $\left\{\left(
\begin{array}{rl}
  1 & d \\
 0 & 1
\end{array}
\right) \right\}_{d \in \Z}$  of  $GL(2,\Z)$
to the automorphism  group $\mathop{\rm Aut}(G)$. The first named author proved in~\cite{Osi1} that this homomorphism can be extended to an injective homomorphism $GL(2, \Z)  \hookrightarrow \mathop{\rm Aut}(G)$.
 By means of the last homomorphism it was also proved in~\cite{Osi1} that  the group $\mathop{\rm Aut}(G)$ is isomorphic to   ${ (\Z \oplus \Z)  \rtimes GL(2,\Z)}$. We note that other  isomorphism $\mathop{\rm Aut}(G) \simeq { (\Z \oplus \Z)  \rtimes GL(2,\Z)}$ was  constructed earlier  by P.~Kahn in~\cite{Ka} by another methods.
\end{nt}

Now, using  the homomorphism $\Z
\to \mathop{\rm Aut}(G)$ given by formula~\eqref{expl-auto}, by the usual way
we construct the group $\tilde{G}$ that is  the semidirect product $G \rtimes \Z$. Explicitly, elements of $\tilde{G}$
are integer quadruples $(a,b,c,d)$, where $(a,b,c) \in G$ and $d \in \Z$. From formula~\eqref{expl-auto}
we obtain the following explicit multiplication law in $\tilde{G}$:
\begin{multline*}
(a_1, b_1, c_1, d_1)  (a_2, b_2, c_2, d_2)= \big( (a_1, b_1, c_1) \, R_{d_1}((a_2,b_2,c_2)), d_1 +d_2   \big) = \\ =
\left( ( a_1, b_1, c_1 ) \left(a_2 + d_1 b_2, b_2, c_2 + \frac{b_2 (b_2 -1) d_1}{2}\right), d_1 + d_2 \right) = \\ =
\left(a_1 +a_2 + d_1b_2, b_1 + b_2, c_1 + c_2 + \frac{b_2(b_2 -1)d_1}{2} + a_1 b_2, d_1 +d_2              \right)  \mbox{.}
\end{multline*}

\begin{nt} \em
From explicit multiplication law in $\tilde{G}$ we see  that the group $\tilde{G}$ can be also considered as a group of the following matrices with integer entries:
 $$
\left(
\begin{array}{cccc}
1 & d & a & c \\
0 & 1 & b &  \frac{b (b-1)}{2}\\
0 & 0 & 1 &  b\\
0 & 0 & 0 & 1
\end{array}
\right) \mbox{,}
$$
From this description it is evident that the group $\tilde{G}$ is nilpotent of class at most $3$.
We consider elements $x = (0,0,0,1)$, $y =(0,1,0,0)$ and $z = (0,1,0,0)$ of the group $\tilde{G}$.
Then we have $[x,y]=(1,0,-1,0)$ and $[[x,y],z]= (0,0,1,0)$. Therefore
the length of the lower central series of $\tilde{G}$  equals $3$. Consequently,
the group $\tilde{G}$ is nilpotent of class  $3$.
\end{nt}

The group $K^*$ acts on the group $K$ by multiplication. Evidently, for any $d \in \Z$, $f \in K^*$, $h \in K$ we have
$R_d(f \cdot h)= R_d(f) \cdot R_d(h)$. Hence and from construction of the actions,  for any $g \in \widehat{K^*}$ and any $S \in \D_{\oo}(K)$,
any $T \in \D'_{\oo}(K) $ we have
\begin{equation}  \label{im-for}
R_d(g \cdot S) = R_d(g)  \cdot R_d(S) \mbox{,} \qquad \qquad  R_d(g \cdot T) = R_d(g) \cdot R_d(T)  \mbox{.}
\end{equation}
Therefore we have that $R_d \left(\D'_{\oo}(K)^{{\oo'}^*} \right)= \D'_{\oo}(K)^{{\oo'}^*}$ for any $d \in \Z$. From formulas~\eqref{im-for}
we have also that $R_d (g \cdot (R_{d}^{-1} (T))) = R_d(g)\cdot T$. The last formula implies that for any $d \in \Z$, any $(a,b,c)  \in G$ and any $V \in \D'_{\oo}(K)^{{\oo'}^*} $
we have
\begin{equation}  \label{im-for2}
R_d ((a,b,c) \cdot (R_{d}^{-1} (V))) = R_d((a,b,c))\cdot V \mbox{.}
\end{equation}
From formula~\eqref{im-for2} we obtain that the following action of the group $\tilde{G}$ on the $\C$-vector space $\D'_{\oo}(K)^{{\oo'}^*}$ is well-defined, i.e it is an action of the group:
$$
(a,b,c,d) \cdot V = (a,b,c) \cdot (R_d(V))  \mbox{,}
$$
where $(a,b,c,d) \in \tilde{G}$ and $V \in \D'_{\oo}(K)^{{\oo'}^*} $.

We recall that in formula~\eqref{vp}  and Definition~\ref{Psi} we defined elements ${\vp_l \in \D'_{\oo}(K)^{{\oo'}^*}}$ and the $\C$-vector subspace $\Psi$ of the space
$\D'_{\oo}(K)^{{\oo'}^*}$.
\begin{prop}
The space $\Psi$ is a $\tilde{G}$-invariant subspace of the space  $\D'_{\oo}(K)^{{\oo'}^*}$ with the following explicit action
$$
(a,b,c,d) \cdot \vp_l =   q^{-c-a(l+1) -\frac{dl(l+1)}{2}} \cdot  ((a+dl)(\vp))_{b+l}             \mbox{,}
$$
where $(a,b,c, d) \in \tilde{G}$ and $(a+dl)(\vp) \in \D'(\bar{K})^{\oo_{\bar{K}}^*}$ is an application of the action of the element $u^{a+dl} \in \bar{K}^*$
to the element $\vp \in \D'(\bar{K})^{\oo_{\bar{K}}^*}$.
\end{prop}
\begin{proof}
By the definition of the action, direct computation and Proposition~\ref{exten} we have
\begin{multline*}
(a,b,c,d) \cdot \vp_l = (a,b,c) \cdot ( R_d ((0,l,0) \cdot \vp_0)  ) = (a,b,c) \cdot ( R_d((0,l,0)) \cdot R_d(\vp_0) )= \\ =
\left( (a,b,c) \left(dl, l, \frac{dl(l-1)}{2}\right)  \right) \cdot \vp_0  = \left(a+dl, b+l, c + \frac{dl(l-1)}{2} + al \right) \cdot \vp_0= \\=
q^{-c-a(l+1) -\frac{dl(l+1)}{2}} \cdot  ((a+dl)(\vp))_{b+l} \mbox{.}
\end{multline*}
The proposition is proved.
\end{proof}

\smallskip

From this proposition for any $\alpha \in C^*$ we obtain that
$$
(a,b,c,d) \cdot (g_{\alpha})_l =  q^{-c-a(l+1) -\frac{dl(l+1)}{2}} \alpha^{a+dl}  (g_{\alpha})_{l+b}  \mbox{,}
$$
where elements $g_{\alpha}$ were constructed in Remark~\ref{example1}.
Hence, by direct computation from the last formula we obtain the following theorem.
\begin{Th} \label{Theorem-tr}
For any $\alpha \in \C^*$ the subspace $\beta(\Psi_{\alpha})$ of the space $\Psi$ is $\tilde{G}$-invariant. We have the traces $\tr_{\alpha}$ of elements of $\tilde{G}$
in $\beta(\Psi_{\alpha})$ with respect to a distinguished basis  $ (g_{\alpha})_l = \beta( \Delta(l, g_{\alpha}))$, where $l \in \Z$ (see discussion in the beginning of this section):
\begin{itemize}
\item $\tr_{\alpha}((a,b,c,d))= 0$ if and only if $b \ne 0$;
\item $\tr_{\alpha}(a,0,c,d) = q^{-c-a} \alpha^{a} \sum\limits_{l \in \Z} q^{-al -\frac{d l(l+1)}{2}} \alpha^{dl}$.
\end{itemize}
\end{Th}

\begin{nt}{\em
Clearly, the series for $\tr_{\alpha}(a,0,c,d)$ is a convergent series when ${d > 0}$. Moreover,
$\tr_{\alpha}(a,0,c,d) = q^{-c-a} \alpha^{a} \sum\limits_{l \in \Z} \tilde{q}^{l^2} \eta^l $,
where $\eta= q^{-a} q^{-\frac{d}{2}} \alpha^d  $, $\tilde{q} = q^{-\frac{d}{2}}$. The series $\sum\limits_{l \in \Z} \tilde{q}^{l^2} \eta^l$
is the classical Jacobi theta function (compare also with~\cite[Example]{P1}).}
\end{nt}

It is easy to see that $\tr_{q\alpha} (a,o,c,d)= \alpha^d \tr_{\alpha} (a,0,c,d)$. Since the traces are invariants of a $\tilde{G}$-representation $\beta(\Psi_{\alpha})$   with respect to a distinguished basis, from Theorems~\ref{main} and~\ref{Theorem-tr} we obtain a proposition.
\begin{prop}
The family of $\tilde{G}$-representations $\beta(\Psi_{\alpha})$ parametrized  by elements $\alpha$ from the set
$\C^* \setminus  \{1\} = \mathbb{P}^1(\C)  \setminus \{0, 1, \infty\}$   is a family of pairwise non-isomorphic irreducible $\tilde{G}$-representations.
\end{prop}

\vspace{0.3cm}

\noindent
D. V. Osipov \\ \medskip
\noindent Steklov Mathematical Institute of Russsian Academy of Sciences, ul. Gubkina 8, Moscow, 119991 Russia \\  \medskip
\noindent National University of Science and Technology ``MISiS'',  Leninsky Prospekt 4, Moscow,  119049 Russia \\   \medskip
\noindent {\it E-mail:}  ${d}_{-} osipov@mi.ras.ru$

\bigskip

\noindent
A. N. Parshin \\   \medskip
\noindent Steklov Mathematical Institute of Russsian Academy of Sciences, ul. Gubkina 8, Moscow, 119991 Russia \\   \medskip
\noindent {\it E-mail:}  $parshin@mi.ras.ru$

\end{document}